\documentclass [12pt]{amsart}
\usepackage{geometry}
\usepackage{graphicx}
\usepackage{amsthm}
\usepackage{amssymb}

\theoremstyle{plain}
\newtheorem{thm}{Theorem}[section]
\newtheorem{cor}[thm]{Corollary}
\newtheorem{lem}[thm]{Lemma}
\newtheorem{prop}[thm]{Proposition}
\newtheorem{defn}[thm]{Definition}
\newtheorem{exa}[thm]{Example}

\begin{document}

\title [{{ On Graded $s$-Prime Submodules}}]{On Graded $s$-Prime Submodules}

 \author[{{H. Saber}}]{\textit{Hicham Saber}}

\address
{\textit{Hicham Saber, Department of Mathematics, Faculty of Science, University of Ha'il, Saudi Arabia.}}
\bigskip
{\email{\textit{hicham.saber7@gmail.com}}}

\author[{{T. Alraqad}}]{\textit{Tariq Alraqad}}

\address
{\textit{Tariq Alraqad, Department of Mathematics, Faculty of Science, University of Ha'il, Saudi Arabia.}}
\bigskip
{\email{\textit{t.alraqad@uoh.edu.sa}}}

 \author[{{R. Abu-Dawwas }}]{\textit{Rashid Abu-Dawwas }}

\address
{\textit{Rashid Abu-Dawwas, Department of Mathematics, Yarmouk
University, Jordan.}}
\bigskip
{\email{\textit{rrashid@yu.edu.jo}}}


\date{}

\begin{abstract}In this article, we introduce the concepts of graded $s$-prime submodules  which is a
generalization of graded prime submodules.
We study the behavior of this notion  with respect to graded homomorphisms, localization of graded modules, direct product, and idealization.
  We succeeded to prove the existence of graded $s$-prime submodules in the case of graded-Noetherian modules.
Also, we provide some sufficient conditions for the existence of such objects in the general case, as well as, in the particular case of grading by $\mathbb{Z}$,  a finite group, or a polycyclic-by-finite group, in addition to crossed product grading.
\end{abstract}

 \maketitle

 \section{Introduction}

Throughout this article, $G$ will be a group with identity $e$, $R$ a commutative ring with a nonzero unity $1$ and $M$ an $R$-module. $R$ is said to be $G$-graded if $R=\displaystyle\bigoplus_{g\in G} R_{g}$ with $R_{g}R_{h}\subseteq R_{gh}$ for all $g, h\in G$ where $R_{g}$ is an additive subgroup of $R$ for all $g\in G$. The elements of $R_{g}$ are called homogeneous of degree $g$. If $x\in R$, then $x$ can be written as $\displaystyle\sum_{g\in G}x_{g}$, where $x_{g}$ is the component of $x$ in $R_{g}$.
 Also, we set $h(R)=\displaystyle\bigcup_{g\in G}R_{g}$, and $h^*(R)=h(R)\setminus\{0\}$. Moreover, it has been proved in \cite{Nastasescue} that $R_e$ is a subring of $R$ and $1\in R_e$. Let $I$ be an ideal of a graded ring $R$. Then $I$ is said to be graded ideal if $I=\displaystyle\bigoplus_{g\in G}(I\cap R_{g})$, i.e., for $x\in I$, $x=\displaystyle\sum_{g\in G}x_{g}$ where $x_{g}\in I$ for all $g\in G$. An ideal of a graded ring need not be graded. Let $R$ be a $G$-graded ring and $I$ is a graded ideal of $R$. Then $R/I$ is $G$-graded by $\left(R/I\right)_{g}=(R_{g}+I)/I$ for all $g\in G$.

Assume that $M$ is an unital $R$-module. Then $M$ is said to be $G$-graded if $M=\displaystyle\bigoplus_{g\in G}M_{g}$ with $R_{g}M_{h}\subseteq M_{gh}$ for all $g,h\in G$ where $M_{g}$ is an additive subgroup of $M$ for all $g\in G$. The elements of $M_{g}$ are called homogeneous of degree $g$. It is clear that $M_{g}$ is an $R_{e}$-submodule of $M$ for all $g\in G$. Moreover, we set $h(M)=\displaystyle\bigcup_{g\in G}M_{g}$. Let $N$ be an $R$-submodule of a graded $R$-module $M$. Then $N$ is said to be graded $R$-submodule if $N=\displaystyle\bigoplus_{g\in G}(N\cap M_{g})$, i.e., for $x\in N$, $x=\displaystyle\sum_{g\in G}x_{g}$ where $x_{g}\in N$ for all $g\in G$. An $R$-submodule of a graded $R$-module need not be graded. Let $M$ be a $G$-graded $R$-module and $N$ be a graded $R$-submodule of $M$. Then $M/N$ is a graded $R$-module by $\left(M/N\right)_{g}=(M_{g}+N)/N$ for all $g\in G$.

\begin{lem}(\cite{Farzalipour}, Lemma 2.1)\label{1.3} Let $R$ be a $G$-graded ring and $M$ be a $G$-graded $R$-module.

\begin{enumerate}

\item If $I$ and $J$ are graded ideals of $R$, then $I+J$ and $I\bigcap J$ are graded ideals of $R$.

\item If $N$ and $K$ are graded $R$-submodules of $M$, then $N+K$ and $N\bigcap K$ are graded $R$-submodules of $M$.

\item If $N$ is a graded $R$-submodule of $M$, $r\in h(R)$, $x\in h(M)$ and $I$ is a graded ideal of $R$, then $Rx$, $IN$ and $rN$ are graded $R$-submodules of $M$.
\end{enumerate}
\end{lem}

If $N$ is a graded $R$-submodule of $M$, then $Ann_{R}(N)=\left\{r\in R:rN=\{0\}\right\}$ is a graded ideal of $R$ (see \cite{Farzalipour2}), and $(N:_{R}M)=\left\{r\in R:rM\subseteq N\right\}$ is a graded ideal of $R$ (see \cite{Atani}).

 The concept of graded prime submodules is one of the pillar stones of the theory of graded modules. For years, there have been many studies and generalizations on this notions. See, for example, \cite{Dawwas Zoubi Bataineh}, \cite{Atani}, \cite{Atani Farzalipour}, \cite{Farzalipour}, and \cite{Oral Tekir Agargun}.  Recall from \cite{Atani} that a graded prime $R$-submodule is a proper graded $R$-submodule $N$ of $M$ having the property that $rm\in N$ implies $r\in (N :_{R} M)$ or $m\in N$ for each $r\in h(R)$ and $m\in h(M)$.  Here we denote the set of all graded prime $R$-submodules by $GSpec(_{R}M)$, in particular, we write $GSpec(R)$ to express the set of all graded prime ideals of $R$.
 Also, A graded $R$-module $M$ is called a graded multiplication $R$-module if $N = (N :_{R} M)M$ for every graded $R$-submodule $N$ of $M$ (see \cite{Escoriza}). If the only graded $R$-submodules of $M$ are $\{0\}$ and $M$ itself, then we call $M$ a graded simple $R$-module (see \cite{Nastasescue}).

  The aim of this article is to introduce the notion of  graded $s$-prime submodules, which is in a certain sense, the closest extension of  the class of graded prime submodules.
  Recall that a subset $S$ of $R$ is called a multiplicatively closed subset (briefly, m.c.s.) of $R$ if:
     \begin{enumerate}
       \item $0\notin S$, and  $1\in S$,
       \item $ st\in S$ for all $s, t\in S$,
     \end{enumerate}. (see \cite{Wang}). Note that $S_{I} = h(R)-I$ is a m.c.s. of $R$ for every $I\in GSpec(R)$.
   The concept of $S$-prime submodules was introduced in \cite{Sevim}, where a submodule $N$ of an  $R$-module $M$ is called  $S$-prime if  $(N :_{R} M)\bigcap S = \emptyset$, and  there exists $s\in S$ such that whenever $rm\in N$ then  $sr\in (N :_{R} M)$ or $sm\in N$ for each $r\in R$ and $m\in h(M)$.  Our original goal was to investigate this notion, as it is, in the setting of graded modules. But we noticed that the core of this new concept which is build on the following property:
  \begin{center}
  If $rm\in N$ then  $sr\in (N :_{R} M)$ or $sm\in N$ for each $r\in h(R)$ and $m\in h(M)$,
  \end{center}
    only requires the existence of the element $s$ to establish most of the results obtained in \cite{Sevim}, and  not the whole set $S$.
   In addition, we found that imposing the condition  on $S$ to be a m.c.s. is a tight assumption,  as it is shown in the following example, causing the exclusion of a wide range of submodules.
   \begin{exa}
     Let $T$ be an integral domain, and $R=T[X]/<X^2>$. Then zero ideal in $R$ can not be an $S$-prime for m.c.s. $S$ in $R$. Otherwise, if $\overline{X}$ denotes the class of $X$ in $R$, then $\overline{X}\overline{X}= \overline{0}$, and so there exists  $s=\overline{t}\in S$ such that $\overline{t}\overline{X}= \overline{0}$, i.e.;  $t X\in <X^2> $. Hence $t \in <X> $, and therefore  $s^2=\overline{t^2}=0$, which implies that $0\in S$, a contradiction.
   \end{exa}
    Also, from another point of view, we are looking to introduce a concept which is, in a certain sense, the closest extension of the class of graded prime submodules, meaning coming with a notion encompassing  submodules failing to be prime, but are "almost prime", such as $<X^2>$ in $R[X]$.\\

     Based on these remarks, and other technical details, which will be more clear  throughout the proofs of our results, we choose to define the new idea of graded $s$-prime submodules. More precisely, let   $s\in h^*(R)$ be a nonzero homogenous element of $R$, and  $M$ be a graded $R$-module. Then a graded $R$-submodule $N$ of $M$  is said to be a gr`aded $s$-prime $R$-submodule if  $s\notin(N :_{R} M)$,  and whenever $rm\in N$ then
  $sr\in (N :_{R} M)$ or $sm\in N$ for each $r\in h(R)$ and $m\in h(M)$. In particular, a graded ideal $I$ of $R$ is called a graded $s$-prime ideal if $I$ is a graded $s$-prime $R$-submodule of the $R$-module $R$.
  The set of all graded $s$-prime $R$-submodules is denoted by $GSpec_s(_{R}M)$, and we write $GSpec_s(R)$ to express the set of all graded $s$-prime ideals of $R$. Notice that, every graded prime $R$-submodule of $M$ is graded $s$-prime for each homogenous element $s\notin(N :_{R} M)$, however, the converse is not true in general, see Example \ref{Example 2.3}.\\

  When working in the non-graded case, meaning $G=\{e\}$ is the trivial group, one get  the notion of $s$-prime $R$-submodules, which is clearly a generalization of the concept of  $S$-prime submodules studied  in \cite{Sevim}. To prove the efficiency of our new idea,  Sections \ref{sec 1} and \ref{sec 2} are devoted to recover most of the results in \cite{Sevim}, not only in the classical case but in the full generality of the graded case. Among several results, we  prove that if $N\in GSpec_{S}(_{R}M)$, then $(N :_{R} M)\in GSpec_{S}(R)$, and the converse is not true in general, see  Example \ref{7}. But we show that converse holds in the interesting case of  graded multiplication $R$-module (Proposition \ref{Proposition 2.9}). Also,  we study the behavior of graded $s$-prime submodules with respect to graded homomorphisms, localization of graded modules, direct product, and idealization.\\

  Section \ref{sec 3} is considered as the main part of this article, where we tackle the problem of existence of $s$-prime modules. We succeeded to prove that if $M$ is a graded-Noetherian $R$-module,  then every graded $R$-submodule $N$ of $M$ is $s$-prime for some $s\in h^*(R)$, see Theorem \ref{maximal-2}.  Applying this to the case of trivial grading, we find that for a commutative ring  $R$ with a nonzero unity $1$, and a Noetherian $R$-module $M$, then for any $R$-submodule  $N$ of $M$ there exists $s\in R$ such that $N\in Spec_s(_{R}M)$, here $Spec_s(_{R}M)$ stands for  the set of all $s$-prime submodules of $M$, see Theorem\ref{maximal-2}.
  These are  a consequence of a more general result relating the existence $s$-prime modules to the maximality, with respect to inclusion, of the graded $R$-submodule $(N:_{M}t)=\left\{m\in M: tm\in N\right\}$ of $M$ in the  set $\mathfrak{D}_N=\{(N :_{M} t),\,  s\in  h^*(R)\setminus (N :_{R} M)\}$, see Theorem\ref{maximal}. Similar results are provided in the case of grading by $\mathbb{Z}$,  a finite group, or polycyclic-by-finite group. \\
  
  In the last section we treat the interesting case when  $(R,G)$ is a crossed product grading.  We succeeded to prove that if  $I=\displaystyle\bigoplus_{g\in G}I_{g}$  is a graded ideal of $R$, then  $I \in GSpec_{s}(R)$, for some $s\in h^*(R)$, if and only if $I_e \in GSpec_{t}(R)$, for some $t\in R_e$, see Theorem \ref{crossed-product}. In particular, if  $G$ is any group,  and $R$   is a Noatherian commutative ring with a nonzero unity $1$, then any  graded ideal of $R[G]$ is $s$-prime  for some $s\in h^*(R)$, see Theorem \ref{group ring}.

\section{Graded $s$-Prime Submodules}\label{sec 1}

In this section, we introduce and study the concept of graded $s$-prime submodules.

\begin{defn}Let $s\in h^*(R)$ be a nonzero homogenous element of $R$, and let $M$ be a graded unital $R$-module. A graded $R$-submodule $N$ of $M$  is said to be a graded $s$-prime $R$-submodule if  $s\notin(N :_{R} M)$,  and whenever $rm\in N$ then  $sr\in (N :_{R} M)$ or $sm\in N$ for each $r\in h(R)$ and $m\in h(M)$.
\end{defn}

As a direct result, we have:
\begin{lem}\label{direct} Let $M$ be a $G$-graded $R$-module,  $N$ be a graded $R$-submodule of $M$, and $s\in h^*(R)$.
\begin{enumerate}
\item If $N\in GSpec(_{R}M)$,  then $N\in GSpec_{s}(_{R}M)$.\\

\item For $t\in HU(R)$, we have $N\in GSpec_s(_{R}M)$ if and only if $N\in GSpec_{ts}(_{R}M)$.

\end{enumerate}

\end{lem}

The next example shows that the converse of Lemma \ref{direct} (1) is not true in general.

\begin{exa}\label{Example 2.3}Consider $R=\mathbb{Z}$ and $G=\mathbb{Z}_{2}$. Then $R$ is trivially $G$-graded by $R_{0}=\mathbb{Z}$ and $R_{1}=\{0\}$. Consider the $R$-modules $T=\mathbb{Z}[i]$ and $L=\mathbb{Z}_{2}[i]$. Then $T$ and $L$ are $G$-graded by $T_{0}=\mathbb{Z}$, $T_{1}=i\mathbb{Z}$, $L_{0}=\mathbb{Z}_{2}$ and $L_{1}=i\mathbb{Z}_{2}$. So, $M=T\times L$ is a $G$-graded $R$-module where $M_{0}=T_{0}\times L_{0}$ and $M_{1}=T_{1}\times L_{1}$. Now, $N=\{0\}\times\{0\}$ is a graded $R$-submodule of $M$ with $(N:_{R}M)=\{0\}$. If we put $s=2$, then $s\in h^*(R)$. Now we show that $N\in GSpec_{s}(_{R}M)$. Let $r\in h(R)$ and $m\in h(M)$ such that $rm\in N$.

\underline{\textbf{Case 1}:} If $m\in M_{0}$, then $m=(t_{0}, l_{0})$ for some $t_{0}\in \mathbb{Z}$ and $l_{0}\in \mathbb{Z}_{2}$, and then $rm=(rt_{0}, \overline{rl_{0}})\in N$, which implies that $rt_{0}=0$ and $\overline{rl_{0}}=0$. If $r=0$, then we are done. Assume that $t_{0}=0$. Then $sm=2(t_{0}, l_{0})\in N$. Hence, $N\in GSpec_{s}(_{R}M)$.

\underline{\textbf{Case 2}:} If $m\in M_{1}$, then $m=(t_{1}, l_{1})$ for some $t_{1}\in i\mathbb{Z}$ and $l_{1}\in i\mathbb{Z}_{2}$, and then $t_{1}=ia$ and $l_{1}=ib$ for some $a\in \mathbb{Z}$ and $b\in \mathbb{Z}_{2}$. So, $rm=(ira, i\overline{rb})\in N$, which implies that $ra=0$ and $\overline{rb}=0$. Hence, by Case (1), $N\in GSpec_{s}(_{R}M)$.

On the other hand, $2\in h(R)$ and $(0, \overline{1})\in h(M)$ such that $2(0, \overline{1})\in N$, but $2\notin (N:_{R}M)$ and $(0, \overline{1})\notin N$. Hence, $N$ is not graded prime $R$-submodule of $M$.
\end{exa}
$ $
\begin{prop}\label{Lemma 2.5}Let $M$ be a $G$-graded $R$-module, $N$ be a graded $R$-submodule of $M$ and $s\in h^*(R)$. Then $N\in GSpec_{s}(_{R}M)$ if and only if $IK\subseteq N$ implies $sI\subseteq (N :_{R} M)$ or $sK\subseteq N$ for each graded ideal $I$ of $R$ and graded $R$-submodule $K$ of $M$.
\end{prop}

\begin{proof}Suppose that $N\in GSpec_{s}(_{R}M)$. Suppose that $IK\subseteq N$ for some graded ideal $I$ of $R$ and graded $R$-submodule $K$ of $M$.. Assume that $sK\nsubseteq N$. Then there exists $k\in K$ such that $sk\notin N$, and then there exists $g\in G$ such that $sk_{g}\notin N$. Note that, $k_{g}\in K$ as $K$ is graded submodule. Let $r\in I$. Then $r_{h}\in I$ for all $h\in G$ as $I$ is graded ideal. Now, $r_{h}k_{g}\in IK\subseteq N$ for all $h\in G$. Since $N\in GSpec_{s}(_{R}M)$ and $sk_{g}\notin N$, we have $sr_{h}\in (N:_{R}M)$ for all $h\in G$, and then $sr\in (N:_{R}M)$. Hence, $sI\subseteq(N:_{R}M)$.
Conversely, let $r\in h(R)$ and $m\in h(M)$ with $rm\in N$. Now, $I = Rr$ is a graded ideal of $R$ and $K = Rm$ is a graded $R$-submodule of $M$ such that $IK\subseteq N$. Then by assumption, $sI\subseteq (N :_{R} M)$ or $sK\subseteq N$, and so either $sr\in (N :_{R} M)$ or $sm\in N$. Therefore, $N\in GSpec_{s}(_{R}M)$.
\end{proof}

\begin{cor}\label{Corollary 2.6}Let $R$ be a graded ring, $P$ be a graded ideal of $R$ and $s\in h^*(R)$. Then $P\in GSpec_{s}(R)$ if and only if $IJ\subseteq P$ implies $sI\subseteq P$ or $sJ\subseteq P$ for each graded ideals $I$ and $J$ of $R$.
\end{cor}

\begin{prop}\label{Proposition 2.9} Let $M$ be a graded $R$-module and $s\in h^*(R)$.
 If $N\in GSpec_{s}(_{R}M)$, then $(N :_{R} M)\in GSpec_{s}(R)$.

\end{prop}

\begin{proof}
 By (\cite{Atani}, Lemma 2.1), $(N:_{R}M)$ is a graded ideal of $R$. Let $ab\in (N :_{R} M)$ for some $a, b\in h(R)$. Then we have $ab\in h(R)$ and $abm\in N$ for all $m\in h(M)$. If $sa\in (N :_{R} M)$, there is nothing to prove. Suppose that $sa\notin (N :_{R} M)$. Since $N\in GSpec_{s}(_{R}M)$, $sbm\in N$ for all $m\in h(M)$. Now, Let $t\in M$ and suppose that  $t=\displaystyle\sum_{g\in G}t_{g}$ where $t_{g}\in M_{g}$ for all $g\in G$, and then $sbt_{g}\in N$ for all $g\in G$, which implies that $sbt=sb\left(\displaystyle\sum_{g\in G}t_{g}\right)=\displaystyle\sum_{g\in G}sbt_{g}\in N$, so that $sb\in (N :_{R} M)$. Therefore, $(N :_{R} M)\in GSpec_{s}(R)$.
\end{proof}

 The next example shows that the converse of the above result is not true in general.

\begin{exa}\label{7}Consider $R=\mathbb{Z}$ and $G=\mathbb{Z}_{4}$. Then $R$ is trivially $G$-graded by $R_{0}=\mathbb{Z}$ and $R_{1}=R_{2}=R_{3}=\{0\}$. Consider the $R$-module $T=\mathbb{Z}[i]$. Then $T$ is $G$-graded by $T_{0}=\mathbb{Z}$, $T_{2}=i\mathbb{Z}$ and $T_{1}=T_{3}=\{0\}$. So, $M=T\times T$ is a $G$-graded $R$-module where $M_{g}=T_{g}\times T_{g}$ for all $g\in G$. Choose $m=(2, 0)\in h(M)$, then $N=Rm$ is a graded $R$-submodule of $M$ with $(N:_{R}M)=\{0\}\in GSpec(R)$. On the other hand, $N\notin GSpec(_{R}M)$ since $2\in h(R)$ and $(3, 0)\in h(M)$ such that $2(3, 0)\in N$, but $2\notin (N:_{R}M)$ and $(3, 0)\notin N$.
\end{exa}

 The case of graded multiplication $R$-modules is of special interest since we get a positive answer for the converse. Indeed, we have:

\begin{prop}\label{Proposition 2.10}Let $M$ be a graded multiplication $R$-module and $s\in h^*(R)$.
 If  $(N :_{R} M)\in GSpec_{s}(R)$, then $N\in GSpec_{s}(_{R}M)$.
\end{prop}

\begin{proof}
Let $I$ be a graded ideal of $R$ and $K$ be a graded $R$-submodule of $M$ with $IK\subseteq N$. Then we have that $I(K :_{R} M)\subseteq (IK :_{R} M)\subseteq (N :_{R} M)$. Since $(N :_{R} M)\in GSpec_{s}(R)$, by Corollary \ref{Corollary 2.6}, we have $sI\subseteq (N :_{R} M)$ or $s(K :_{R} M)\subseteq (N :_{R}M)$. Thus, we have that $sI\subseteq (N :_{R} M)$ or $sK = s(K :_{R} M)M\subseteq (N :_{R} M)M = N$. By Proposition \ref{Lemma 2.5}, $N\in GSpec_{s}(_{R}M)$.

\end{proof}

Recall that for a graded multiplication $R$-module $M$,  the product of two graded $R$-submodules $N, K$ of $M$ is defined by $NK = (N :_{R} M)(K :_{R} M)M$ (see \cite{Escoriza}). As a consequence of Proposition \ref{Proposition 2.9} and Proposition \ref{Lemma 2.5}, we have the following specific result.

\begin{cor}\label{Corollary 2.10}Suppose that $M$ is a graded multiplication $R$-module, $s\in h^*(R)$, and $N$ is a graded $R$-submodule of $M$ with $s\notin(N :_{R} M)$. Then $N\in GSpec_{s}(_{R}M)$ if and only for every graded $R$-submodules $L, K$ of $M$ with $LK\subseteq N$, we have $sL\subseteq N$ or $sK\subseteq N$.
\end{cor}

\begin{prop}\label{Proposition 2.12} Let $M$ be a graded multiplication $R$-module, and let  $N\in GSpec_{s}(_{R}M)$ for some $s\in h^*(R)$. If  $K, L$ are graded $R$-submodules of $M$ such that $K\bigcap L\subseteq N$ . Then $sK\subseteq N$ or $sL\subseteq N$.
\end{prop}

\begin{proof} Suppose that $sL\nsubseteq N$. Then $sm\notin N$ for some $m\in L$, and then there exists $g\in G$ such that $sm_{g}\notin N$, where $m_{g}\in L$ as $L$ is graded submodule. Let $r\in (K :_{R} M)$. Then $r_{h}\in(K:_{R}M)$ for all $h\in G$ as $(K:_{R}M)$ is a graded ideal, and then $r_{h}m_{g}\in (K :_{R} M)L\subseteq L\bigcap K\subseteq N$. Since $N\in GSpec_{s}(_{R}M)$ and $sm_{g}\notin N$, we have that $sr_{h}\in (N :_{R} M)$ for all $h\in G$, and then $sr\in (N:_{R}M)$, so that $s(K :_{R} M)\subseteq (N :_{R} M)$. Since $M$ is a graded multiplication $R$-module, $sK = s(K :_{R} M)M\subseteq (N :_{R} M)M =N$.
\end{proof}

Let $M$ be a graded $R$-module. Then the set of all homogeneous zero divisors of $M$ is $HZ(M)=\left\{r\in h(R):rm=0\mbox{ for some nonzero }m\in h(M)\right\}$.

\begin{thm}\label{Proposition 2.27}Let $M$ be a finitely generated graded $R$-module and $s\in h^*(R)$. If every proper graded $R$-submodule of $M$ is graded $s$-prime, then  we have:
\begin{enumerate}
\item  $HZ(M)=Ann_{R}(M)\bigcap h(R)$.\\
\item If $t$ is a homogenous element such that $t \notin HZ(M)$, then $t M= M$.
\end{enumerate}
\end{thm}

\begin{proof}
\begin{enumerate}
 \item Let $r\in HZ(M)$. Then $r\in h(R)$ and there is a nonzero $m\in M$ with $rm = 0$. Since the graded zero submodule is graded $s$-prime and $rm = 0$, we have $sr\in Ann_{R}(M)$ or $sm = 0$. If $sm = 0$, then $s\in Ann_{R}(m)$.
     Now, let $K = Ann_{R}(m)M$, then $K$ is a graded $R$-submodule of $M$ such that $s\in (K :_{R} M)$, and then $K = Ann_{R}(m)M = M$. By (\cite{Attiyah}, Corollary 2.5), $1-x\in Ann_{R}(M)\subseteq Ann_{R}(m)$ for some $x\in Ann_{R}(m)$, implying that $Ann_{R}(m) = R$ and so $m = 0$, which is a contradiction.
      Therefore, $sr\in Ann_{R}(M)$ and hence $s\in (Ann_{M}(r) :_{R} M)$. Since $Ann_{M}(r)=(0:_{M}r)$ is a graded $R$-submodule of $M$, by Lemma \ref{5}, we must have $Ann_{M}(r) = M$, i.e.; $r\in Ann_{R}(M)$. Thus, $HZ(M) = Ann_{R}(M)\bigcap h(R)$.

\item Suppose that $t$ is a homogenous element such that $t \notin HZ(M)$. If $t^2 M=M$, then  $M=t^2M \subseteq tM$, and so $t M= M$.
If $t^2 M\neq M$, then $t^2 M$ is  $s$-prime with $(t)(t M)\subseteq t^2 M$, hence by Proposition \ref{Lemma 2.5}, we have $st M\subseteq t^2 M$.
Using the fact that $t \notin HZ(M)$, we deduce that $s M \subseteq t M$, i.e.; $s\in (t M:_{R} M)$, and hence $tM=M$ as desired.
 \end{enumerate}
\end{proof}

\begin{cor}\label{cor6.1}Let $M$ be a finitely  generated graded $R$-module, $N$ be a graded $R$-submodule of $M$, and $s\in h^*(R)$. Suppose that  every proper graded $R$-submodule of $M$ is  graded $s$-prime. Then we have $N=M$ if $( N :_{R} M) \neq Ann_{R}(M)$.
\end{cor}

\begin{proof}
Let $x\in ( N :_{R} M)\setminus Ann_{R}(M)$. By the above theorem,  we have $M=xM$, and so $M=xM \subseteq ( N :_{R} M) M\subseteq N$. Hence $N=M$.
\end{proof}

Since for a graded multiplication $R$-module $( N :_{R} M)=( 0 :_{R} M)= Ann_{R}(M)$ if and only if $N=0$, we have:

\begin{cor}\label{cor 6.2}Let $M$ be a finitely  generated  graded multiplication $R$-module, and $s\in h^*(R)$. If every proper graded $R$-submodule of $M$ is  graded $s$-prime. Then $M$ is a simple graded $R$-module.
\end{cor}

In particular, we get the following result:

\begin{cor}\label{cor 6.3}Let $R$ be a graded ring, and $s\in h^*(R)$. If every proper graded ideal of $R$ is  graded $s$-prime. Then $R$ is a graded field.
\end{cor}

\section{Behavior of graded $s$-prime submodules }\label{sec 2}

In the section we study the behavior of graded $s$-prime submodules with respect to graded homomorphisms, localization of graded modules, direct product, and idealization.

Let $M$ and $T$ be $G$-graded $R$-modules. Then an $R$-homomorphism $f:M\rightarrow T$ is said to be a graded $R$-homomorphism if $f(M_{g})\subseteq T_{g}$ for all $g\in G$ (see \cite{Nastasescue}).

\begin{prop}\label{Proposition 2.7}Let $M$ and $T$ be $G$-graded $R$-modules, and $s\in h^*(R)$. Assume that $f:M\rightarrow T$ is a graded $R$-homomorphism.
\begin{enumerate}
\item If $K\in GSpec_{s}(_{R}T)$ and $s\notin(f^{-1}(K) :_{R} M)$, then $f^{-1}(K)\in GSpec_{s}(_{R}M)$.

\item If $f$ is a graded $R$-epimorphism and $N\in GSpec_{s}(_{R}M)$ with $Ker(f)\subseteq N$, then $f(N)\in GSpec_{s}(_{R}T)$.
\end{enumerate}
\end{prop}

\begin{proof}
\begin{enumerate}
\item Clearly, $f^{-1}(K)$ is a graded $R$-submodule of $M$. Let $rm\in f^{-1}(K)$ for some $r\in h(R)$, $m\in h(M)$. Then $f(rm) = rf(m)\in K$. Since $K\in GSpec_{s}(_{R}T)$, we have $sr\in (K :_{R} T)$ or $sf(m) = f(sm)\in K$. Now we show that $(K :_{R} T)\subseteq (f^{-1}(K) :_{R} M)$. Let $x\in (K :_{R}T)$. Then $xT\subseteq K$. Since $f(M)\subseteq T$, we have that $f(xM) = xf(M)\subseteq xT\subseteq K$, which implies that $xM\subseteq xM + Ker(f) = f^{-1}(f(xM))\subseteq f^{-1}(K)$ and hence $x\in (f^{-1}(K) :_{R} M)$. As $(K :_{R} T)\subseteq (f^{-1}(K) :_{R} M)$, we have either $sr\in (f^{-1}(K) :_{R} M)$ or $sm\in f^{-1}(K)$. Hence, $f^{-1}(K)\in GSpec_{s}(_{R}M)$.

\item Clearly, $f(N)$ is a graded $R$-submodule of $T$. We have $s\notin(f(N) :_{R} T)$. Otherwise, $f(sM) = sf(M)\subseteq sT\subseteq f(N)$, and so  $sM\subseteq sM + Ker(f)\subseteq N + Ker(f) =N$, which implies that $sM\subseteq N$. Thus $s\in(N :_{R} M)$ , which contradicts that $N\in GSpec_{s}(_{R}M)$.
    Now, let $r\in h(R)$, $t\in h(T)$ with $rt\in f(N)$. Since $f$ is an $R$-epimorphism, there is an $m\in h(M)$ such that $t = f(m)$, and so  $rt = rf(m) = f(rm)\in f(N)$. Using the fact that $Ker(f)\subseteq N$ , we get $rm\in N$. But $N\in GSpec_{s}(_{R}M)$ and $(N :_{R} M)\subseteq (f(N) :_{R} T)$, imply that  $sr\in (f(N) :_{R} T)$ or $st = sf(m) = f(sm) \in f(N)$. Thus $f(N)\in GSpec_{s}(_{R}T)$.
\end{enumerate}
\end{proof}

For the sake of completeness we give the proof of the graded version of the following two classical result.

\begin{lem}\label{2} Let $M$ be a graded $R$-module, $L$ be a graded $R$-submodule of $M$, and $N$ be an $R$-submodules of $M$ such that $L\subseteq N$. Then $N$ is a graded $R$-submodule of $M$ if and only if $N/L$ is a graded $R$-submodule of $M/L$.
 \end{lem}

 \begin{proof} Suppose that $N$ is a graded $R$-submodule of $M$. Clearly, $N/L$ is an $R$-submodule of $M/L$. Let $x+L\in N/L$. Then $x\in N$ and since $N$ is graded, $x=\displaystyle\sum_{g\in G}x_{g}$ where $x_{g}\in N$ for all $g\in G$ and then $(x+L)_{g}=x_{g}+L\in N/L$ for all $g\in G$. Hence, $N/L$ is a graded $R$-submodule of $M/L$. Conversely, let $x\in N$. Then $x=\displaystyle\sum_{g\in G}x_{g}$ where $x_{g}\in M_{g}$ for all $g\in G$ and then $(x_{g}+L)\in (M_{g}+L)/L=\left(M/L\right)_{g}$ for all $g\in G$ such that $\displaystyle\sum_{g\in G}(x+L)_{g}=\displaystyle\sum_{g\in G}(x_{g}+L)=\left(\displaystyle\sum_{g\in G}x_{g}\right)+L=x+L\in N/L$. Since $N/L$ is graded, $x_{g}+L\in N/L$ for all $g\in G$ which implies that $x_{g}\in N$ for all $g\in G$. Hence, $N$ is a graded $R$-submodule of $M$.
 \end{proof}

\begin{prop}\label{Corollary 2.8}Let $M$ be a graded $R$-module, $L$ be a graded $R$-submodule of $M$, and $s\in h^*(R)$.
\begin{enumerate}
\item If $K\in GSpec_{s}(_{R}M)$ and $s\notin(K :_{R} L)$, then $L\bigcap K\in GSpec_{s}(_{R}L)$.

\item Suppose that $N$ is an $R$-submodule of $M$ with $L\subseteq N$. Then $N\in GSpec_{s}(_{R}M)$ if and only if $N/L\in GSpec_{s}(_{R}(M/L))$.
\end{enumerate}
\end{prop}

\begin{proof}
\begin{enumerate}
\item Clearly, $L\bigcap K$ is a graded $R$-submodule of $L$. Consider the graded $R$-homomorphism $f : L\rightarrow M$ defined by $f(m) = m$ for all $m\in L$. Then  $f^{-1}(K) = L\bigcap K$. Also, we have  $s\notin (f^{-1}(K) :_{R} L)$. Otherwise we get $sL\subseteq f^{-1}(K) = L\bigcap K\subseteq K$ and thus $s\in (K :_{R} L)\bigcap S$, which is a contradiction. The result holds by Proposition \ref{Proposition 2.7} (1).

\item Suppose that $N\in GSpec_{s}(_{R}M)$ with $L\subseteq N$. Then by Lemma \ref{2}, $N/L$ is a graded $R$-submodule of $M/L$. Consider the graded $R$-epimorphism $f : M\rightarrow M/L$ defined by $f(m) = m + L$ for all $m\in M$. By Proposition \ref{Proposition 2.7} (2), $N/L\in GSpec_{s}(_{R}(M/L))$. Conversely, by Lemma \ref{2}, $N$ is a graded $R$-submodule of $M$. Let $rm\in N$ for some $r\in h(R)$, $m\in h(M)$. Then $r(m+ L) = rm+ L\in N/L$. Since $N/L\in GSpec_{s}(_{R}(M/L))$, we have $sr\in (N/L :_{R} M/L) = (N :_{R} M)$ or $s(m + L) = sm + L\in N/L$. Therefore, we have $sr\in (N :_{R} M)$ or $sm\in N$. Hence, $N\in GSpec_{s}(_{R}M)$.
\end{enumerate}
\end{proof}


Let $M$ be a $G$-graded $R$-module, and $N$ be a graded $R$-submodule of $M$. If  $s$ and $t$ are elements of $h^*(R)$ such that $N$ is both $s$-prime and $t$-prime,  then one can directly deduce that $N$ is also $st$-prime whenever  $st\notin (N :_{R} M )$. This have been said, then it is natural to focus on graded modules of fractions with respect to multiplicative sets lying outside $(N :_{R} M)$.

Consider a nonempty subset $S$ of $R$. We call $S$ a multiplicatively closed subset (briefly, m.c.s.) of $R$ if (i) $0\notin S$, (ii) $1\in S$, and (iii) $ss^{\prime}\in S$ for all $s, s^{\prime}\in S$ (see \cite{Wang}). Note that $S_{I} = h(R)-I$ is a m.c.s. of $R$ for every $I\in GSpec(R)$.
Let $S\subseteq h(R)$ be a m.c.s and $M$ be a graded $R$-module. Then $S^{-1}M$ is a graded $S^{-1}R$-module with $(S^{-1}M)_g=\left\{\frac{m}{s},m\in M_h, s\in S\cap R_{hg^{-1}}\right\}$ and $(S^{-1}R)_g=\left\{\frac{a}{s},a\in R_h, s\in S\cap R_{hg^{-1}}\right\}$ for all $g\in G$. It is obvious that \\$S^{*}=\left\{x\in h(R):\frac{x}{1}\mbox{ is a homogeneous unit of }S^{-1}R\right\}$ is a m.c.s of $R$ containing $S$.

\begin{prop}\label{Proposition 2.2} Let $M$ be a graded $R$-module, and $N$ be a graded $R$-submodule of $M$. Suppose that  $S\subseteq h(R)$ is a m.c.s verifying $(N :_{R} M)\bigcap S=\emptyset$.
 If $N\in GSpec_{s}(_{R}M)$ for some $s\in S$, then $S^{-1}N$ is a graded prime $S^{-1}R$-submodule of $S^{-1}M$.

\end{prop}

\begin{proof}
 Suppose that $N\in GSpec_{s}(_{R}M)$ for some $s\in S$. Let $\frac{r}{a}\in h(S^{-1}R)$ and $\frac{m}{b}\in h(S^{-1}M)$ such that $\frac{r}{a}.\frac{m}{b}\in S^{-1}N$. Then $urm\in N$ for some $u\in S$. Since $N\in GSpec_{s}(_{R}M)$, we have $s ur\in (N:_{R}M)$ or $s m\in N$, which implies that $\frac{r}{a}=\frac{s ur}{s ua}\in S^{-1}(N:_{R}M)\subseteq(S^{-1}N:_{S^{-1}R}S^{-1}M)$ or $\frac{m}{b}=\frac{s m}{s t}\in S^{-1}N$. Hence, $S^{-1}N$ is a graded prime $S^{-1}R$-submodule of $S^{-1}M$.

\end{proof}

Now, we are going to prove that the converse of Proposition \ref{Proposition 2.2} is not true in general. Firstly, we need the introduce following:

\begin{defn}A graded commutative ring $R$ with unity is said to be a graded field if every nonzero homogeneous element of $R$ is unit.
\end{defn}

The next example shows that a graded field need not be a field.

\begin{exa}Let $R$ be a field and suppose that $F=\left\{x+uy:x, y\in R, u^{2}=1\right\}$. If $G=\mathbb{Z}_{2}$, then $F$ is $G$-graded by $F_{0}=R$ and $F_{1}=uR$. Let $a\in h(F)$ such that $a\neq0$. If $a\in F_{0}$, then $a\in R$ and since $R$ is a field, we have $a$ is a unit element. Suppose that $a\in F_{1}$. Then $a=uy$ for some $y\in R$. Since $a\neq0$, we have $y\neq0$, and since $R$ is a field, we have $y$ is a unit element, that is $zy=1$ for some $z\in R$. Thus, $uz\in F_{1}$ such that $(uz)a=uz(uy)=u^{2}(zy)=1.1=1$, which implies that $a$ is a unit element. Hence, $F$ is a graded field. On the other hand, $F$ is not a field since $1+u\in F-\{0\}$ is not a unit element since $(1+u)(1-u)=0$.
\end{exa}

\begin{lem}\label{1}Let $R$ be a graded field and $M$ be a graded $R$-module. Then every proper graded $R$-submodule of $M$ is graded prime.
\end{lem}

\begin{proof}Let $N$ be a proper graded $R$-submodule of $M$. Suppose that $r\in h(R)$ and $m\in h(M)$ such that $rm\in N$. If $r=0$, then $r=0\in (N:_{R}M)$. If $r\neq0$, then $r$ is unit as $R$ is graded field, and then $r^{-1}\in R$ with $r^{-1}(rm)=m\in N$. Hence, $N$ is a graded prime $R$-submodule of $M$.
\end{proof}

The required example is:

\begin{exa}\label{Example 2.4}Consider $R=\mathbb{Z}$ and $G=\mathbb{Z}_{2}$. Then $R$ is trivially $G$-graded by $R_{0}=\mathbb{Z}$ and $R_{1}=\{0\}$. Consider the $R$-module $T=\mathbb{Q}[i]$. Then $T$ is $G$-graded by $T_{0}=\mathbb{Q}$ and $T_{1}=i\mathbb{Q}$. So, $M=T\times T$ is a $G$-graded $R$-module where $M_{0}=T_{0}\times T_{0}$ and $M_{1}=T_{1}\times T_{1}$. Now, $N=\mathbb{Z}\times\{0\}$ is a graded $R$-submodule of $M$ with $(N:_{R}M)=\{0\}$. Consider $S=\mathbb{Z}-\{0\}$ is a m.c.s of $h(R)$. Then $S^{-1}R=Q$ is a graded field, and then by Lemma \ref{1}, $S^{-1}N$ is a graded prime $S^{-1}R$-submodule of $S^{-1}M$. On the other hand, assume that $s\in S$, and choose a prime number $p$ with $gcd(s, p)=1$. Then $p\in h(R)$ and $(\frac{1}{p}, 0)\in h(M)$ such that $p(\frac{1}{p}, 0)=(1, 0)\in N$, but $sp\notin (N:_{R}M)$ and $s(\frac{1}{p}, 0)=(\frac{s}{p}, 0)\notin N$. Hence, $N$ is not graded $S$-prime $R$-submodule of $M$.
\end{exa}

The following result gives a partial affirmative answer for the converse of Proposition \ref{Proposition 2.2}.

\begin{prop}\label{Proposition 2.17}Suppose that $M$ is a finitely generated graded $R$-module, $S\subseteq h(R)$ is a m.c.s, and $N$ is a graded $R$-submodule of $M$ such that $(N :_{R} M)\bigcap S =\emptyset$. For $s\in S$, we have $N\in GSpec_{s}(_{R}M)$ if and only if $S^{-1}N$ is a graded prime $R$-submodule of $S^{-1}M$ and  $(N :_{M} t)\subseteq (N :_{M} s)$ for all $t\in S$.
\end{prop}

\begin{proof}Suppose that $N\in GSpec_{S}(_{R}M)$. Then the result holds by Lemma \ref{direct} and Proposition \ref{Proposition 2.2}. Conversely, let $r\in h(R)$ and $m\in h(M)$ with $rm\in N$. Then $\frac{r}{1}\frac{m}{1}\in S^{-1}N$. Since $S^{-1}N$ is a graded prime $R$-submodule of $S^{-1}M$ and $M$ is finitely generated, we have that $\frac{r}{1}\in (S^{-1}N :_{S^{-1}R} S^{-1}M) = S^{-1}(N :_{R} M)$ or $\frac{m}{1}\in S^{-1}N$. Then $ur\in (N :_{R} M)$ or $vm\in N$ for some $u, v\in S$. By assumption, we have $(N :_{M} t)\subseteq (N :_{M} s)$ for all $t\in S$. Therefore, if $ur\in (N :_{R} M)$ then $rM\subseteq (N :_{M} u)\subseteq (N :_{M} s)$ and thus $sr\in (N :_{R} M)$. If $vm\in N$ , a similar argument proves that $sm\in N$. Therefore, $N\in GSpec_{s}(_{R}M)$.
\end{proof}


Let $R_{1}$ and $R_{2}$ be $G$-graded rings. Then $R=R_{1}\times R_{2}$ is a $G$-graded ring with $R_{g}=(R_{1})_{g}\times(R_{2})_{g}$ for all $g\in G$ (see \cite{Nastasescue}).

\begin{lem}\label{3}Let $R_{1}$ and $R_{2}$ be $G$-graded rings and $R=R_{1}\times R_{2}$. Then $P=P_{1}\times P_{2}$ is a graded ideal of $R$ if and only if $P_{1}$ is a graded ideal of $R_{1}$ and $P_{2}$ is a graded ideal of $R_{2}$.
\end{lem}

\begin{proof}Suppose that $P$ is a graded ideal of $R$. Clearly, $P_{1}$ is an ideal of $R_{1}$ and $P_{2}$ is an ideal of $R_{2}$. Let $a\in P_{1}$. Then $(a, 0)\in P$, and then $(a, 0)_{g}=(a_{g}, 0)\in P$ for all $g\in G$ as $P$ is graded ideal, which implies that $a_{g}\in P_{1}$ for all $g\in G$. Hence, $P_{1}$ is a graded ideal of $R_{1}$. Similarly, $P_{2}$ is a graded ideal of $R_{2}$. Conversely, it is clear that $P$ is an ideal of $R$. let $(a, b)\in P$. Then $a\in P_{1}$ and $b\in P_{2}$, and then $a_{g}\in P_{1}$ and $b_{g}\in P_{2}$ for all $g\in G$ as $P_{1}$ and $P_{2}$ are graded ideals, which implies that $(a, b)_{g}=(a_{g}, b_{g})\in P$ for all $g\in G$. Hence, $P$ is a graded ideal of $R$.
\end{proof}

\begin{lem}\label{Lemma 2.13}Let $R_{1}$ and $R_{2}$ be $G$-graded rings, $R = R_{1}\times R_{2}$ and $s = (s_{1}, s_{2})\in h^*(R)$. Suppose that $P = P_{1}\times P_{2}$ is a graded ideal of $R$. Then $P\in GSpec_{s}(R)$ if and only if $ P_{1}\in GSpec_{s_{1}}(R_{1})$ and $ s_2\in P_2$ or  $P_{2}\in GSpec_{s_{2}}(R_{2})$ and $s_1\in P_1$.
\end{lem}

\begin{proof}Suppose that $P\in GSpec_{S}(R)$. By Lemma \ref{3}, $P_{1}$ is a graded ideal of $R_{1}$ and $P_{2}$ is a graded ideal of $R_{2}$. Now, $(1, 0), (0, 1)\in h(R)$ with $(1, 0)(0, 1) = (0, 0)\in P$, and so  $s(1, 0) = (s_{1}, 0)\in P$ or $s(0, 1) = (0, s_{2})\in P$. Thus $ s_{1}\in P_{1}$ or $ s_{2}\in P_{2}$.
 We may assume that $ s_{1}\in P_{1}$. Since  or $ s \notin P$, we have $ s_{2}\notin P_{2}$. Let $ab\in P_{2}$ for some $a, b\in h(R_{2})$. Since $(0, a)(0, b)\in P$ and $P\in GSpec_{S}(R)$ , we have either $s(0, a) = (0, s_{2}a)\in P$ or $s(0, b) = (0, s_{2}b)\in P$ and then $s_{2}a\in P_{2}$ or $s_{2}b\in P_{2}$. Therefore, $P_{2}\in GSpec_{S_{2}} (R_{2})$.
 The other case can be treated in the same way.

  Conversely, assume that $ s_{1}\in P_{1}$ and $P_{2}\in GSpec_{s_{2}} (R_{2})$, and so $s\notin P$.  Let $(a, b)(c, d) = (ac, bd)\in P$ for some $a, c\in h(R_{1})$ and $b, d\in h(R_{2})$. Then $bd\in P_{2}$ and so  $s_{2}b\in P_{2}$ or $s_{2}d\in P_{2}$. Now we have  $s(a, b) = (s_{1}a, s_{2}b)\in P$ or
$s(c, d) = (s_{1}c, s_{2}d)\in P$. Therefore, $P\in GSpec_{S}(R)$.  Likewise, one can show that $P\in GSpec_{S}(R)$.
\end{proof}

Let $M_{1}$ be a $G$-graded $R_{1}$-module, $M_{2}$ be a $G$-graded $R_{2}$-module and $R=R_{1}\times R_{2}$. Then $M=M_{1}\times M_{2}$ is $G$-graded $R$-module with $M_{g}=(M_{1})_{g}\times(M_{2})_{g}$ for all $g\in G$ (see \cite{Nastasescue}). Similarly to Lemma \ref{3}, one can prove the following:

\begin{lem}\label{4}Let $M_{1}$ be a $G$-graded $R_{1}$-module, $M_{2}$ be a $G$-graded $R_{2}$-module, $R=R_{1}\times R_{2}$ and $M=M_{1}\times M_{2}$. Then $L=N\times K$ is a graded of $R$-submodule of $M$ if and only if $N$ is a graded $R_{1}$-submodule of $M_{1}$ and $K$ is a graded $R_{2}$-submodule of $M_{2}$.
\end{lem}

\begin{prop}\label{Theorem 2.14}Suppose that $M = M_{1}\times M_{2}$ is a graded $R$-module and $s = (s_{1}, s_{2})\in h^*(R)$, where $M_{i}$ is a $G$-graded $R_{i}$-module, $R = R_{1}\times R_{2}$. Assume that $L = N\times K$ is a graded $R$-submodule of $M$. Then $L\in GSpec_{S}(_{R}M)$ if and only if $N\in GSpec_{s_{1}} (_{R_{1}}M_{1})$ and $s_{2}\in(K :_{R_{2}} M_{2})$ or  $K\in GSpec_{s_{2}} (_{R_{2}}M_{2})$ and $s_{1}\in(N :_{R_{1}} M_{1})$.
\end{prop}

\begin{proof}Suppose that $L\in GSpec_{s}(_{R}M)$. By Lemma \ref{4}, $N$ is a graded $R$-submodule of $M_{1}$ and $K$ is a graded $R$-submodule of $M_{2}$. Now, by Proposition \ref{Proposition 2.9} (1), $(L :_{R} M) = (N :_{R_{1}} M_{1})\times(K :_{R_{2}}M_{2})\in GSpec_{S}(R)$ and so by Lemma \ref{Lemma 2.13}, either $s_{1}\in(N :_{R_{1}} M_{1})$ or $s_{2}\in(K :_{R_{2}} M_{2})$.
 We may assume that $s_{1}\in(N :_{R_{1}} M_{1})$.
 Now, we show that $K\in GSpec_{s_{2}} (_{R_{2}}M_{2})$. Let $rm\in K$ for some $r\in h(R_{2})$, $m\in h(M_{2})$. Then $(1, r)\in h(R)$ and $(0, m)\in h(M)$ such that $(1, r)(0, m) = (0, rm)\in L$. Since $L\in GSpec_{s}(_{R}M)$, we have  $s(1, r) = (s_{1}, s_{2}r)\in(L :_{R} M)$ or $s(0, m) = (0, s_{2}m)\in L$. This implies that $s_{2}r\in (K :_{R_{2}} M_{2})$ or $s_{2}m\in K$. Therefore, $K\in GSpec_{s_{2}} (_{R_{2}}M_{2})$. In the other case, it can be similarly proved that $N\in GSpec_{s_{1}} (_{R_{1}}M_{1})$.

Conversely, suppose that $s_{1}\in(N :_{R_{1}} M_{1})$ and $K\in GSpec_{s_{2}} (_{R_{2}}M_{2})$, and so $s\notin (L :_{R} M) = (N :_{R_{1}} M_{1})\times(K :_{R_{2}}M_{2})$.
   Suppose that $(r_{1}, r_{2})(m_{1}, m_{2}) = (r_{1}m_{1}, r_{2}m_{2})\in L$ for some $(r_{1}, r_{2})\in h(R)$, $(m_{1}, m_{2})\in h(M)$. Then $r_{2}m_{2}\in K$, and since $K\in GSpec_{s_{2}} (_{R_{2}}M_{2})$, we have $s_{2}r_{2}\in (K :_{R_{2}} M_{2})$ or $s_{2}m_{2}\in K$. Therefore $s(r_{1}, r_{2}) = (s_{1}r_{1}, s_{2}r_{2})\in(L :_{R} M)$ or $s(m_{1}, m_{2}) = (s_{1}m_{1}, s_{2}m_{2})\in N\times K =L$. Hence, $L\in GSpec_{s}(_{R}M)$. Similarly one can show that if $N\in GSpec_{s_{1}} (_{R_{1}}M_{1})$ and $s_2 \in(K :_{R_{2}} M_{2})$, then $L\in GSpec_{s}(_{R}M)$.
\end{proof}

Now, using induction and Proposition \ref{Theorem 2.14}, we get the following general result:

\begin{prop}\label{Theorem 2.15}Suppose that $M = M_{1}\times...\times M_{n}$ is a graded $R$-module and $s = (s_{1},..., s_{n})\in h^*(R)$, where $M_{i}$ is a $G$-graded $R_{i}$-module, $R = R_{1}\times...\times R_{n}$. Assume that $N = N_{1}\times...\times N_{n}$ is a graded $R$-submodule of $M$. Then $N\in GSpec_{s}(_{R}M)$ if and only if $N_{i}\in GSpec_{s_{i}} (_{R_{i}}M_{i})$ for some $i$, and $s_{j}\in (N_{j} :_{R_{j}} M_{j}) $ for all $j\neq i$.
\end{prop}


Let $M$ be an $R$-module. The idealization $R(+)M=\left\{  (r,m):r\in
R\mbox{ and }m\in M\right\}  $ of $M$ is a commutative ring with componentwise
addition and multiplication; $(x,m_{1})+(y,m_{2})=(x+y,m_{1}+m_{2})$ and
$(x,m_{1})(y,m_{2})=(xy,xm_{2}+ym_{1})$ for each $x,y\in R$ and $m_{1}%
,m_{2}\in M$. Let $G$ be an abelian group and $M$ be a $G$-graded $R$-module.
Then $X=R(+)M$ is $G$-graded by $X_{g}=R_{g}(+)M_{g}$ for all $g\in G$. Note
that, $X_{g}$ is an additive subgroup of $X$ for all $g\in G$. Also, for
$g,h\in G$, $X_{g}X_{h}=(R_{g}(+)M_{g})(R_{h}(+)M_{h})=(R_{g}R_{h},R_{g}%
M_{h}+R_{h}M_{g})\subseteq(R_{gh},M_{gh}+M_{hg})\subseteq(R_{gh}%
,M_{gh})=X_{gh}$ as $G$ is abelian, (see \cite{RaTeShKo}). If $S$ is a m.c.s. of $h(R)$ and $N$ is a graded $R$-submodule of $M$, then
$S(+)N$ is a m.c.s. of $h(R(+)M)$.

\begin{lem}
\label{6}Let $G$ be an abelian group, $M$ be a $G$-graded $R$-module, $P$ be
an ideal of $R$ and $N$ be an $R$-submodule of $M$ such that $PM\subseteq N$.
Then $P(+)N$ is a graded ideal of $R(+)M$ if and only if $P$ is a graded ideal
of $R$ and $N$ is a graded $R$-submodule of $M$.
\end{lem}

\begin{proof}
Follows from \cite[Proposition 3.3]{RaTeShKo}.
\end{proof}

\begin{prop}\label{Proposition 2.21}Let $G$ be an abelian group, $M$ be a $G$-graded ring, $s\in h^*(R)$, and $P$ be a graded ideal of $R$ with $s\notin P$. Then the following are equivalent:
\begin{enumerate}
\item $P\in GSpec_{s}(R)$.

\item $P(+)M\in GSpec_{(s,m)}(R(+)M)$, for every $m\in h(M)$.

\item $P(+)M\in GSpec_{(s,0)}(R(+)M)$.
\end{enumerate}
\end{prop}

\begin{proof}$(1)\Rightarrow(2):$  Suppose that $P\in GSpec_{s}(R)$. Let $(x, a)(y, b) = (xy, xb + ya)\in P(+)M$ for some $x, y\in h(R)$, $a, b\in h(M)$. Then we have $xy\in P$. Since $P\in GSpec_{s}(R)$, we get $sx\in P$ or $sy\in P$, and so $(s,m)(x, a) \in P(+)M$ or $(s,m)(y, b) =\in P(+)M$. Therefore, $P(+)M\in GSpec_{S(+)\{0\}}(R(+)M)$.

$(2)\Rightarrow(3):$ Trivial.

$(3)\Rightarrow(1):$ Suppose that $P(+)M\in GSpec_{(s,0)}(R(+)M)$. Let $xy\in P$ for some $x, y\in h(R)$. Then $(x, 0)(y, 0)\in P(+)M$. Since $P(+)M\in GSpec_{(s,0)}(R(+)M)$, we have $(s,0)(x, 0) =(sx, 0)\in P(+)M$ or $(s,0)(y,  0) = (sy, 0)\in P(+)M$, and hence  $sx\in P$ or $sy\in P$. Thus, $P\in GSpec_{s}(R)$.
\end{proof}

\section{Existence of graded $s$-prime submodules} \label{sec 3}

In this section we will tackle the problem of existence of $s$-prime modules.

\begin{lem}\label{5} Let $M$ be a $G$-graded $R$-module and $N$ be a graded $R$-submodule of $M$. If $r\in h(R)$, then $(N:_{M}r)=\left\{m\in M: rm\in N\right\}$ is a graded $R$-submodule of $M$.
\end{lem}

\begin{proof} Clearly, $(N:_{M}r)$ is an $R$-submodule of $M$. Let $m\in (N:_{M}r)$. Then $rm\in N$. Now, $m=\displaystyle\sum_{g\in G}m_{g}$ where $m_{g}\in M_{g}$ for all $g\in G$. Since $r\in h(R)$, $r\in R_{h}$ for some $h\in G$ and then $rm_{g}\in R_{h}M_{g}\subseteq M_{hg}\subseteq h(M)$ for all $g\in G$ such that $\displaystyle\sum_{g\in G}rm_{g}=r\left(\displaystyle\sum_{g\in G}m_{g}\right)=rm\in N$. Since $N$ is graded, $rm_{g}\in N$ for all $g\in G$ which implies that $m_{g}\in (N:_{M}r)$ for all $g\in G$. Hence, $(N:_{M}r)$ is a graded $R$-submodule of $M$.
\end{proof}

\begin{lem}\label{Lemma 2.16}Let $M$ be a $G$-graded $R$-module, and $N$ be a graded $R$-submodule of $M$. Suppose that  $s, t\in h^*(R)$ such that $st\notin (N :_{R} M )$ and $N\in GSpec_{s}(_{R}M)$. Then we have the following:
\begin{enumerate}
\item $(N :_{M} t)\subseteq (N :_{M} s)$.\\

\item $((N :_{R} M) :_{R} t)\subseteq ((N :_{R} M) :_{R} s)$.
\end{enumerate}
\end{lem}

\begin{proof}
\begin{enumerate}
\item Let $m\in (N :_{M} t)$. Then $m_{g}\in (N:_{M}t)$ for all $g\in G$ as $(N:_{M}t)$ is graded submodule by Lemma \ref{5}, and then $tm_{g}\in N$ for all $g\in G$. Since $N\in GSpec_{s}(_{R}M)$ and $st\notin (N:_{R}M)$, $sm_{g}\in N$ for all $g\in G$, and then $sm\in N$, that is $m\in (N:_{M}s)$.
\item Directly  follows from   (1).
\end{enumerate}
\end{proof}

\begin{prop}\label{Theorem 2.18}Let $M$ be a $G$-graded $R$-module, $N$ be a graded $R$-submodule of $M$, and $s\in h^*(R)$.
\begin{enumerate}
\item If $(N :_{M} s)\in GSpec(_{R}M)$, then $N\in GSpec_{s}(_{R}M)$.\\

\item Suppose that  $s^2\notin (N :_{R} M )$. If  $N\in GSpec_{s}(_{R}M)$ then $(N :_{M} s)\in GSpec(_{R}M)$.
\end{enumerate}

\end{prop}

\begin{proof} \begin{enumerate}
\item
Assume that $(N :_{M} s)\in GSpec(_{R}M)$, and let $rm\in N$ for some $r\in h(R)$, $m\in h(M)$. Since $rm\in (N :_{M} s)$ and $(N :_{M} s)\in GSpec(_{R}M)$, we have $r\in ((N :_{M} s) :_{R} M)$ or $m\in (N :_{M} s)$, and then $rs\in (N :_{R} M)$ or $sm\in N$. Hence, $N\in GSpec_{S}(_{R}M)$.\\

 \item
 Suppose that $s^2\notin (N :_{R} M )$ and  $N\in GSpec_{s}(_{R}M)$. By Lemma \ref{5}, $(N:_{M}s)$ is a graded $R$-submodule of $M$. Let $r\in h(R)$, $m\in h(M)$ with $rm\in (N :_{M} s)$. Then $(sr)m\in N$. Since $N\in GSpec_{s}(_{R}M)$ , we have $s^{2}r\in (N :_{R} M)$ or $sm\in N$. If $sm\in N$, then there is nothing to prove. Suppose that $sm\notin N$. Then $s^{2}r\in (N :_{R} M)$ and so $r\in ((N :_{R} M) :_{R} s^{2}) \subset ((N :_{R} M) :_{R} s))$ by Lemma \ref{Lemma 2.16}. Hence, we have that $r\in ((N :_{M} s) :_{R} M)$ and so $(N :_{M} s)\in GSpec(_{R}M)$.
\end{enumerate}

\end{proof}

Now,  let $\mathfrak{D}_N=\{(N :_{M} t),\,  s\in  h^*(R)\setminus (N :_{R} M)\}$, notice that $(N :_{M} s)= M$ if and only if $s\in (N :_{R} M)$.
Since $N=(N :_{M} 1)$, $\mathfrak{D}_N$ is nonempty and we have:

\begin{thm}\label{maximal} Let $M$ be a graded $R$-module, and $N$ be a graded $R$-submodule of $M$.
 If $(N :_{M} s)$ is a maximal element in $\mathfrak{D}_N$, with respect to inclusion, then $(N :_{M} s)\in GSpec(_{R}M)$, and thus  $N\in GSpec_{s}(_{R}M)$

\end{thm}

\begin{proof}
Suppose that $(N :_{M} s)$ is a maximal element in $\mathfrak{D}_N$, with respect to inclusion. Let $rm\in N$ for some $r\in h(R)$, $m\in h(M)$.
Since $(N :_{M} s)\subseteq (N :_{M} sr)$, then by maximality we have $(N :_{M} sr)=M$ or $(N :_{M} s)=(N :_{M} sr)$.
If $(N :_{M} sr)=M$, then $sr M \subseteq N$, and so $rM \subseteq (N :_{M} s)$. Hence $r\in ((N :_{M} s) :_{R} M) $.
In case $(N :_{M} s)=(N :_{M} sr)$, then $m\in (N :_{M} r) \subseteq (N :_{M} sr)= (N :_{M} s)$.

\end{proof}
Recall  that $M$ is said to be  graded-Noetherian $R$-module if $M$ satisfies the ascending, respectively graded-Artinian
descending chain condition for graded submodules of M.
It is straightforward to verify that M is graded-Noetherian if and only if every graded
submodule is finitely generated, and also if and only if each non-empty family of graded submodules of M has a maximal element.
Similarly, $M$ is  graded-Artinian if and only if each non-empty family of graded submodules of $M$ has a minimal element, see \cite{Nastasescue}).
Note that an Artinian module is always Noetherian, see \cite{Attiyah}.

 In particular, If $M$ is a graded-Noetherian $R$-module, then $\mathfrak{D}_N$ always has a maximal element, hence we have:

 \begin{thm}\label{maximal-2} Let $M$ be a graded-Noetherian $R$-module, and $N$ be a graded $R$-submodule of $M$. Then there exists $s\in h^*(R)$ such that $N\in GSpec_{s}(_{R}M)$.
\end{thm}
Applying this to the case of trivial grading, we find:

 \begin{thm}\label{maximal-3}  Let $R$ a commutative ring with a nonzero unity $1$, and $M$ be a Noetherian $R$-module.
  Then for any $R$-submodule  $N$ of $M$ there exists $s\in R$ such that $N\in Spec_s(_{R}M)$.
\end{thm}

If $R$ is $\mathbb{Z}$-graded, then  a graded $R$-module $M$ is graded-Noetherian if and only if  $M$ is Noetherian as a $R$-module, see \cite{Nastasescue} Theorem 5.4.7. Thus we have:

\begin{cor}\label{Z-gr} Let $R$ be a $\mathbb{Z}$-graded ring, and  $M$ be a graded-$R$-module. If $M$ is an Artinian or a Noetherian $R$-module, then for every  graded $R$-submodule $N$ of $M$ there exists $s\in h^*(R)$ such that $N\in GSpec_{s}(_{R}M)$.
\end{cor}
\begin{exa} \label{R[X]}
  A typical example for $\mathbb{Z}$-graded rings are polynomial ring. Indeed, let $T$ be a commutative ring with a nonzero unity $1$,  if we put
$T_n= T X^n$ for $n\geq0$, and $T_n= 0$ for $n<0$, then $R=T[X]$ is a $\mathbb{Z}$-graded ring.
Since $R[X]$ is Noetherian whenever $R$ is Noetherian, the above result enables us to produce many examples of graded $s$-prime modules.

\end{exa}

Suppose  that $R$ is $G$-graded ring, with $G$ being a finite group, and let $M$ be a graded $R$-module. Then by  \cite{Nastasescue}, Corollary 5.4.3, the following assertions are equivalent:
\begin{enumerate}
  \item $M$ is graded-Noetherian (resp graded-Artinian).
  \item $M$ is Noetherian (resp Artinian) as an $R_e$- module.
  \item $M$ is Noetherian (resp Artinian) in $R$-mod.
\end{enumerate}

 \begin{prop}\label{noath-1} Let $G$ be a finite group, $R$ be a $G$-graded ring, and  $M$ be a graded-$R$-module. If $M$ is an Artinian or a Noetherian $R_e$-module, then for every  graded $R$-submodule $N$ of $M$ there exists $s\in h^*(R)$ such that $N\in GSpec_{s}(_{R}M)$.
\end{prop}

Recall that a  group $G$ is called polycyclic-by-finite if there is a finite series ${e} = G_0 \triangleleft G_1 \triangleleft ... \triangleleft G_n = G$ of subgroups such that each $G_{i-1}$ is normal in $G_i$ and $G_i/G_{i-1}$ is either finite or cyclic for each $i$, \cite{Nastasescue}.

Also, the ring $R$ is called strongly graded of type $G$ if $1\in R_{\sigma}R_{\sigma^{-1}}$ for all $\sigma\in G$, which is equivalent to $R_{\tau}R_{\sigma}=R_{\tau\sigma}$ for all $\tau,\sigma\in G$, see \cite{Nastasescue}.

The importance of this type of grading lies in that fact that a graded-$R$-module $M=\displaystyle\bigoplus_{g\in G}M_{g}$ will be a Noetherian $R$-module if $M_e$ is a Noetherian $R_e$-module. Since a  graded-$R$-module $M$ is graded-Noetherian if it is  Noetherian as an $R$-module, we have the following:
 \begin{prop}\label{poly} Let $G$ be a polycyclic-by-finite group, $R$ be a strongly graded ring of type $G$, and  $M$ be a graded-$R$-module. If $M_e$ is a Noetherian $R_e$-module, then for every  graded $R$-submodule $N$ of $M$ there exists $s\in h^*(R)$ such that $N\in GSpec_{s}(_{R}M)$.
\end{prop}

\section{graded $s$-prime submodules and crossed product grading} \label{sec 4}

Another interesting type of grading is the so-called crossed product. More precisely,  a grading $(R,G)$ is called crossed product over $supp(R,G)$ if for every $g\in supp(R,G)$, $R_g$ contains a unit. By \cite[Proposition 1.7]{Da10}, this is equivalent to the following: for each  $g\in supp(R,G)$, $R_{g}=R_eu_g$ for some unit $u_g\in R_{g}$. For example, consider a commutative ring $T$ with a nonzero unity $1$, and let  $R_n= T X^n$, $n \in \mathbb{Z}$, where $X$ is an indeterminate.
Then the ring of Laurent polynomials  $R=T[X, X^{-1}]$ is a $\mathbb{Z}$-graded ring, and it is clearly crossed product, see  \cite{Nastasescue} for  more examples.

 \begin{thm}\label{crossed-product} Let $(R,G)$ be a crossed product grading, and  $I=\displaystyle\bigoplus_{g\in G}I_{g}$  be a graded ideal of $R$. Then  $I \in GSpec_{s}(R)$, for some $s\in h^*(R)$, if and only if $I_e \in GSpec_{t}(R)$, for some $t\in R_e$.
\end{thm}
\begin{proof}
 Suppose that $I \in GSpec_{s}(R)$ for some $s\in h^*(R)$. Since $R$ is crossed product and $s$ is a homogenous element,  we have $s=t u_g$ for some $t\in R_e$ and a unit $u_g\in R_{g}$. Then by Lemma \ref{direct}, we have  $I \in GSpec_{t}(R)$. Now, let  $ab\in I_e$ for some $a, b\in R_e$, since $I$ is $t$-prime we get $ta\in I$ or $tb\in I$. Hence $ta\in I\cap R_e= I_e$ or $tb\in I\cap R_e= I_e$, and so $I_e \in GSpec_{t}(R)$ as desired.
 For the converse, Assume that  $I_e \in GSpec_{t}(R)$ for some $t\in R_e$, and let  $ab\in I$ where $a, b\in h^*(R)$. Again by the crossed product property, we have $a=a_e u_g $ and $b=b_e u_h $,  for some $a_e, b_e \in R_e$ and  units $u_g\in R_{g}$, $u_h\in R_{h}$.
 Therefore $a_e b_e u_gu_h=ab\in I$, and so $a_e b_e \in I$, hence  $a_e b_e \in I_e$. Since $I_e \in GSpec_{t}(R)$, we get  $ta_e\in I_e$ or $tb_e\in I_e$, which implies that $ta\in I$ or $tb\in I$. Thus $I \in GSpec_{t}(R)$.

\end{proof}

Now, let $G$ be any group, and consider a commutative ring $T$ with a nonzero unity $1$. Then the group ring  $R=T[G]$ is a $G$-graded ring by taking $R_g= T g$, $g \in G$, and $(R,G)$  is clearly a crossed product grading. Since  $R_e=T$, using the above theorem and Theorem \ref{maximal-3} we get the following:

 \begin{thm}\label{group ring} Let $G$ be a group, $R$  be a commutative ring with a nonzero unity $1$. If  $R$ is a Noatherian ring, then any  graded ideal of $R[G]$ is $s$-prime  for some $s\in h^*(R)$.
\end{thm}

\end{document}